\documentclass[12pt]{article}
\usepackage{amssymb,comment}
\usepackage{amsmath}
\usepackage{graphicx}
\usepackage{color}
\usepackage{cite}
\newcounter{foo}
\usepackage{titling}

\usepackage{url}

\newfont{\blb}{msbm10 scaled\magstep1}
\newfont{\comp}{cmr12 scaled\magstep1}
\newfont{\compb}{cmr10 scaled\magstep2}
\newfont{\sbb}{cmssbx10 scaled\magstep3}
\newfont{\sbbb}{cmssbx10 scaled\magstep5}
\newfont{\sbs}{cmssbx10 scaled\magstep1}
\newtheorem{theorem}{Theorem}
\newtheorem{lemma}{Lemma}

\newtheorem{proposition}{Proposition}
\newenvironment{proof}{{\bf Proof.}}{\hfill{ }\vrule height10pt width5pt depth1pt}
\newtheorem{conjecture}[foo]{Conjecture}

\usepackage{hyperref}
\hypersetup{colorlinks = true, linkcolor=black, citecolor = {black}, urlcolor = black}
\title{The asymptotics of $r(4,t)$}

\makeatletter
\renewcommand\@date{{%
  \vspace{-\baselineskip}%
  \large\centering
  \begin{tabular}{@{}c@{}}
    Sam Mattheus\thanks{E-mail: sam.mattheus@vub.be. Research supported by a Fulbright Visiting Scholar Fellowship and a
    Fellowship of the Belgian American Foundation while the author was a Visiting Scholar at University of California San Diego.}   \end{tabular}%
\quad \quad \quad  \begin{tabular}{@{}c@{}}
    Jacques Verstraete\thanks{E-mail: jacques@ucsd.edu. Research supported by the National Science Foundation FRG Award DMS-1952786.} \\
  \end{tabular}

  \bigskip

\begin{tabular}{c c}
\small Department of Mathematics & \small Department of Mathematics\\
\small Vrije Universiteit Brussel & \small University of California, San Diego\\
\small Pleinlaan 2 & \small 9500 Gilman Drive, La Jolla \\
\small 1050 Brussels, Belgium & \small CA 92093-0112, USA. \\
\end{tabular}

\bigskip
\bigskip

  \today
}}
\makeatother

\begin{document}

\maketitle

\begin{abstract}
\noindent For integers $s,t \geq 2$, the Ramsey number $r(s,t)$ denotes the minimum $n$ such that every $n$-vertex
graph contains a clique of order $s$ or an independent set of order $t$. In this paper we prove
\[ r(4,t) = \Omega\Bigl(\frac{t^3}{\log^4 \! t}\Bigr) \quad \quad \mbox{ as }t \rightarrow \infty\]
which determines $r(4,t)$ up to a factor of order $\log^2 \! t$, and solves a  conjecture of Erd\H{o}s.
\end{abstract}

\section{Introduction}

Ramsey Theory is an area of mathematics underpinned by the philosophy that in any large enough structure, there exists a relative large uniform substructure. The area 
is named after F. P. Ramsey~\cite{R}, but has roots in a variety of branches of mathematics,
including logic, set theory, topology, geometry and number theory. Celebrated results include Schur's Theorem~\cite{Schur} leading to Fermat's Last Theorem modulo primes, Rado's partition regularity~\cite{Rado}, van der Waerden's Theorem~\cite{vdW} on arithmetic progressions and Shelah's Theorem~\cite{Shelah}, and Bourgain's Theorem~\cite{B} on Euclidean distortion in metric Ramsey Theory, to mention a few. The area has grown into a cornerstone of modern combinatorics research, and the central quantities of study are known as Ramsey Numbers.

\medskip

The classical expository example is the statement that amongst any six people, there will be at least three people who all know each other, or at least three people who all do not know each other. In general, for integers $s,t \geq 2$, the {\em Ramsey number} $r(s,t)$ denotes the minimum $n$ such that every $n$-vertex graph contains either a clique of order $s$ or an independent set of size $t$, and the afore-mentioned example is equivalent to the statement $r(3,3) \leq 6$.  The quantities $r(s,t)$ are the cornerstone of Ramsey Theory, and have been studied for many decades;
their existence is proved by Ramsey's Theorem~\cite{R} -- see also the book by Graham, Rothschild, Solymosi and Spencer~\cite{GRSS}. The quantities $r(t,t)$ are sometimes referred to as {\em diagonal Ramsey Numbers} -- see Campos, Griffiths, Morris, Sahasrabudhe~\cite{CGMS} for a recent major breakthrough -- whereas $r(s,t)$ for fixed $s$ and $t \rightarrow \infty$ are referred to as {\em off-diagonal Ramsey numbers}. The focus of this paper is on off-diagonal Ramsey numbers. 

\medskip

The original upper bounds on off-diagonal Ramsey numbers were given by Erd\H{o}s and Szekeres~\cite{ES} in 1935. They showed that for all fixed $s \geq 3$ and $t \rightarrow \infty$ we have $r(s,t) = O(t^{s-1})$.
Ajtai, Koml\'{o}s and Szemer\'{e}di~\cite{AKS} established the first improvement to this upper bound on $r(s,t)$ by analyzing a randomized greedy algorithm for producing large independent sets. Bohman and Keevash~\cite{BK10} proved a lower bound by analyzing the random $K_s$-free graph process, improving on earlier results of Spencer~\cite{Spencer1,Spencer2}. These bounds are as follows: for $s \ge 3$, there exists a constants $c_1(s),c_2(s) > 0$ such that the off-diagonal Ramsey numbers satisfy
\begin{equation} \label{rsn}
	c_1(s) \frac{t^{\frac{s+1}{2}}}{(\log t)^{\frac{s + 1}{2} - \frac{1}{s-2}}} \; \; \leq  \; \; r(s,t) \; \; \leq  \; \; c_2(s) \frac{t^{s-1}}{(\log t)^{s - 2}}.
\end{equation}

\medskip

Extending ideas of Shearer \cite{Sh}, the upper bound was further improved by Li, Rousseau and Zang~\cite{LRZ} who showed that as $t \rightarrow \infty$
\begin{equation} \label{rsnupper}
 r(s,t) \; \; \leq \; \; (1 + o(1))\frac{t^{s-1}}{(\log t)^{s - 2}}.
 \end{equation}

\medskip

The only off-diagonal Ramsey numbers $r(s,t)$ for $s \geq 3$ whose order of magnitude
is known is $r(3,t)$, as it was shown in 1995 by Kim~\cite{K} that $r(3,t) = \Omega(t^2/\log t)$ as $t \rightarrow \infty$, matching previous upper bounds
by Ajtai, Koml\'{o}s and Szemer\'{e}di~\cite{AKS} and Shearer~\cite{Sh}, and improving earlier bounds of Spencer~\cite{Spencer1}. The current state of the art is due to Fiz Pontiveros, Griffiths and Morris~\cite{FGM} and Bohman and Keevash~\cite{BK0}, where $r(3,t)$ is determined asymptotically up to a factor four.

\medskip

The current best lower bounds for $r(4,t)$ come from the $K_4$-free process in random graphs, studied by Bohman and Keevash~\cite{BK10},
improving earlier bounds of Spencer~\cite{Spencer2}. With the upper bound (\ref{rsnupper}), the best bounds are, for some absolute constant $a > 0$:
\[ a\frac{t^{\frac{5}{2}}}{\log^2 \! t} \; \; \leq \; \; r(4,t) \; \; \leq \; \; (1 + o(1))\frac{t^3}{\log^2 \! t}.\]

The exponent $5/2$ has stood for more than forty years -- see Spencer~\cite{Spencer1,Spencer2}.

\bigskip

In this paper, we determine $r(4,t)$ up to a factor of order $\log^2 \! t$, and show the exponent is $3$:

\begin{theorem}\label{thm:main}
As $t \rightarrow \infty$,
\[r(4,t) = \Omega\left(\frac{t^3}{\log^4 t}\right).\]
\end{theorem}

This solves a long-standing conjecture of Erd\H{o}s~\cite{Chung}. 
The upper bound for $r(5,t)$ from (\ref{rsn}) is $\widetilde{O}(t^4)$, whereas the lower bound given by (\ref{rsn}) has the same order of magnitude $\widetilde{\Omega}(t^3)$ as the lower bound for $r(4,t)$ in Theorem \ref{thm:main} up to logarithms, leaving the problem of the asymptotics of $r(5,t)$ as a tantalizing open problem.

\bigskip

Theorem \ref{thm:main} also gives almost tight bounds on multicolor Ramsey numbers: for $k \geq 2$
let $r_k(4;t)$ denote the minimum $n$ such that every $k$-coloring of the edges of $K_n$ contains a monochromatic $K_4$ in one of the
first $k - 1$ colors or a monochromatic $K_t$ in the last color. In particular, $r_2(4;t) = r(4,t)$. The upper bound $r_k(4;t) = O(t^{2k-1}/(\log t)^{2k-2})$ was proven by   He and Wigderson~\cite{HeW}, generalizing the result for $r_k(3;t)$ due to Alon and R\"odl~\cite{AR}. Using the approach of Alon and R\"odl~\cite{AR}, we establish a lower bound on $r_k(4;t)$ which is sharp up to polylogarithmic factors:

\begin{theorem} \label{thm:main2}
For each $k \geq 3$, as $t \rightarrow \infty$,
\[ r_k(4;t) = \Omega\left(\frac{t^{2k - 1}}{(\log t)^{6(k-1)}}\right).\]
\end{theorem}


\bigskip

The constructive methods of this paper are inspired by the approaches of Alon and R\"{o}dl~\cite{AR} and Mubayi and the second author~\cite{MV},
and may be useful for providing lower bounds on
other graph Ramsey numbers, for example cycles versus cliques -- see Conlon, Mubayi and the authors~\cite{CMMV}. The constructions in~\cite{C,MV,CMMV}
rely on point-line incidence graphs from finite geometry and random sampling. Our construction for $r(4,t)$ in this paper relies on unitals in finite geometry,
and hence has a substantial non-probabilistic aspect, unlike the afore-mentioned constructions for $r(s,t)$, which rely heavily on random graphs.

\subsection{Organization}

This paper is organized as follows. In Section \ref{sec:pg}, for each prime power $q$, we describe the {\em classical} or {\em Hermitian unitals},
from which we obtain a partial linear space with $q^3 + 1$ points and $n = q^2(q^2 - q + 1)$ lines, which we call {\em secants}, each containing $q + 1$ of the points.

\bigskip

We define in Section 2 the graph $H_q$ whose vertex set is the set of $n$ secants to the unital, and where two secants are adjacent if they intersect in a point of the unital.
Thus $H_q$ is a union of $q^3 + 1$ edge-disjoint cliques of order $q^2$ with $n = q^2(q^2 - q + 1)$ vertices. The graph $H_q$ has the key property (due to O'Nan~\cite{onan})
that all $K_4$s have at least three vertices inside the designated cliques of order $q^2$. What remains is to modify $H_q$ to remove all these $K_4$s while controlling the independence number. The structure of $H_q$ is discussed at length in Section 2.

\bigskip

In Section 3, we construct the random $n$-vertex $K_4$-free graph $H_q^*$, as a union of random complete bipartite subgraphs, one from each maximal clique in $H_q$. This type of ``random block construction'' was first introduced by Brown and R\"{o}dl~\cite{BR}, and considered by Dudek and R\"{o}dl~\cite{DR}, Wolfovitz~\cite{W}, Dudek, Retter and R\"{o}dl~\cite{DRR}, Kostochka, Mubayi and the second author~\cite{KMV}, Conlon~\cite{C}, Mubayi and the second author~\cite{MV}, and Gowers and Janzer~\cite{GJ}. 
The main theorem in Section 3 states that if $q$ is large enough, then all sets of $2^{24}q^2$ vertices of $H_q^*$ induce at least $2^{40} q^{3}$ edges in $H_q^*$ with positive probability. This is proved in Section \ref{sec:hqstar} using the Hoeffding-Azuma inequality. In particular, we address a remark of Conlon on
optimal pseudorandomness of the triangle-free graphs defined in~\cite{C}.

\bigskip

Fixing such an instance $G_q^*$, we use a theorem of Kohayakawa, Lee, R\"{o}dl and Samotij~\cite{KLRS} in Section \ref{sec:container} to show that the number of independent sets of size $t = \lceil 2^{30}q \log^2\!q\rceil$ in $G_q^*$ is at most $(q/\log^2\! q)^t$. This is an alternative but related approach to the spectral approach of Mubayi and the second author~\cite{MV} and Alon and R\"{o}dl~\cite{AR} -- see Samotij~\cite{Sam} for a survey on methods for counting independent sets in graphs, and also see Axenovich, Brada\v{c}, Gishboliner, Mubayi and Weber~\cite{ABGMW} which has similar ideas. Finally, in Section \ref{sec:proofofmain}, by randomly sampling vertices of $G_q^*$ as in~\cite{MV} with probability $(\log^2\! q)/q$, we arrive at a $K_4$-free graph with at least $(q^3 \log^2\! q)/2$ vertices and no independent sets of size $t$, and this proves $r(4,t) \geq c t^3/\log^4\! t$ for some constant $c > 0$. We did not expend effort in optimizing the value of $c$; from the proof $c = 2^{-100}$ will do.

 \bigskip

Finally, in Section \ref{sec:proofofmain2} we use Theorem \ref{thm:main} and the ideas of Alon and R\"{o}dl~\cite{AR}, based on random blowups, to prove Theorem \ref{thm:main2}.

 \bigskip

 We use the following graph-theoretic notation.  Let $G$ be a graph and denote by $V(G)$ and $E(G)$ the vertex set and edge set of $G$ respectively, and $e(G) = |E(G)|$. For $X \subseteq V(G)$, let $G[X]$ denote the subgraph of $G$ induced by $X$, and $e(X) = e(G[X])$ when the graph $G$ is clear from the context.

\section{Unitals and the O'Nan configuration}\label{sec:pg}

A {\em unital} in the projective plane $\mathrm{PG}(2,q^2)$ is a set $\mathcal{U}$ of $q^3 + 1$ points such that every line of $\mathrm{PG}(2,q^2)$ intersects $\mathcal{U}$ in $1$ or $q + 1$ points. Lines will be referred to as {\em tangents} or {\em secants} respectively. A {\em classical} or {\em Hermitian unital} $\mathcal{H}$ is a unital described in homogeneous co-ordinates as the following set of one-dimensional subspaces of $\mathbb F_{q^2}^3$:
\[ \mathcal{H} = \{\left<x,y,z\right> \subset \mathbb F_{q^2}^3 : x^{q + 1} + y^{q + 1} + z^{q + 1} = 0\}.\]
Here arithmetic is in the finite field $\mathbb F_{q^2}$, and $\left<x,y,z\right>$ is the one-dimensional subspace of $\mathbb F_{q^2}^3$ generated by $(x,y,z) \neq 0$. The set $\mathcal{H}$ is the set of {\em absolute points} of a {\em unitary polarity} -- see Barwick and Ebert~\cite{BE} for a monograph.
We may consider the partial linear space whose points are the points of $\mathcal{H}$ and the lines are the secants to $\mathcal{H}$. Combinatorially, the lines form a {\em design} or {\em Steiner $(q + 1)$-tuple system}: every pair of points of $\mathcal{H}$
is contained in exactly one of the lines.

\bigskip

One of the remarkable features of this partial linear space is that it does not contain the
so-called {\em O’Nan configuration}, namely the configuration of
four lines meeting in six points shown in the figure on the left below~\cite{onan}:

\begin{center}
\includegraphics[width=3.5in]{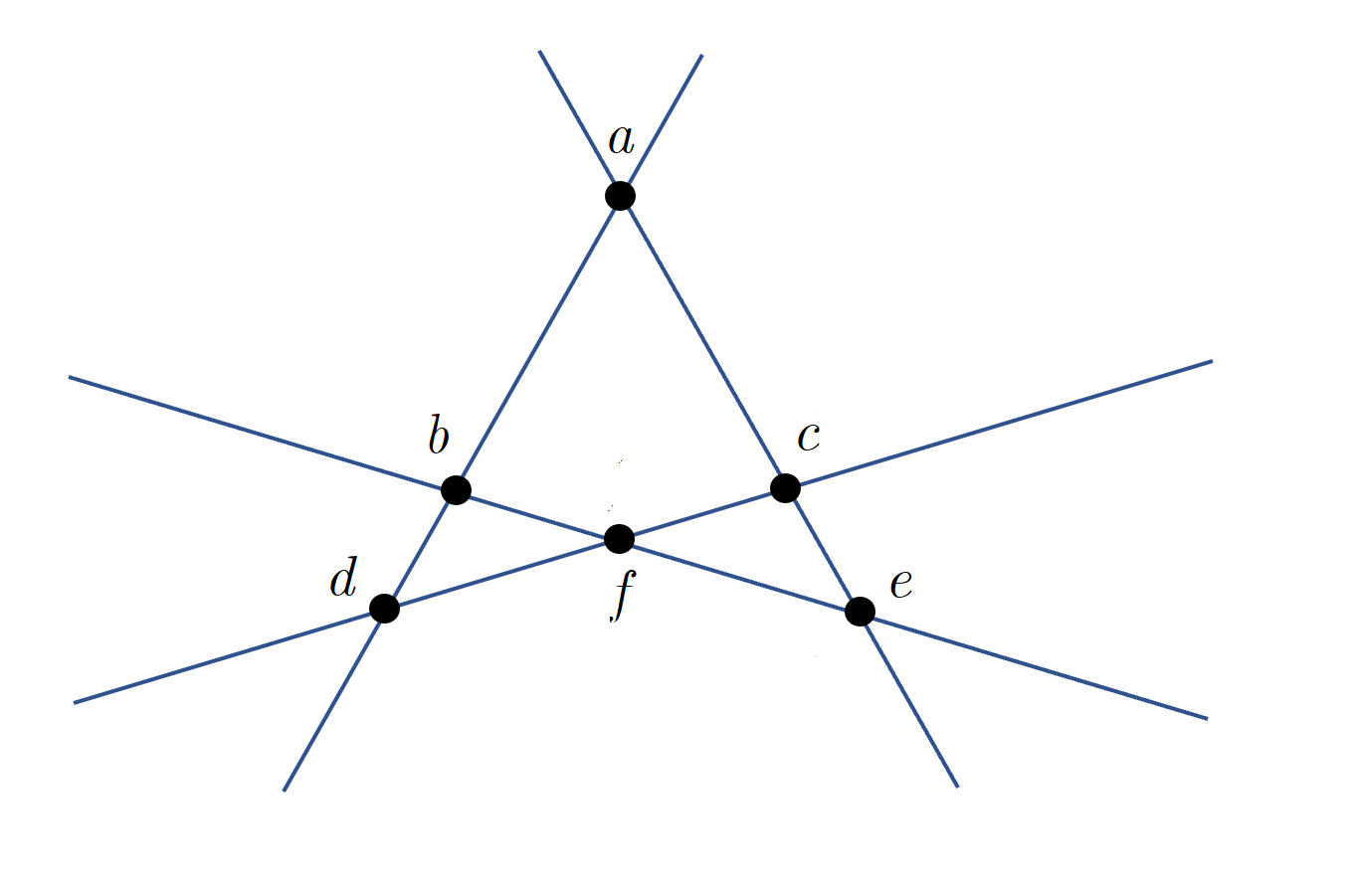}
\vspace{-0.2in}
\begin{center}
O'Nan or Pasch configuration\label{fig:pasch}
\end{center}
\end{center}

\bigskip

In the combinatorial design theoretic literature, this configuration is also referred to as the {\em Pasch configuration}.
Piper~\cite{P} conjectured that the Hermitian unital is characterized among all such Steiner systems
by the absence of O’Nan configurations. We give a short proof here that $\mathcal{H}$ does not contain these configurations, first proved
by O'Nan~\cite{onan}:

\bigskip

\begin{proposition}\label{thm:no-onan}
The Hermitian unital does not contain O'Nan configurations.
\end{proposition}

\begin{proof}
Label the points $a,b,c,d,e,f$ of the O'Nan configuration as shown in the figure above, where a point $a$ is identified with a chosen generator $(a_1,a_2,a_3) \in \mathbb F_{q^2}^3 \backslash \{0\}$. For $a,b \in \mathbb F_{q^2}^3 \backslash \{0\}$, define
\[ \sigma(a,b) = a_1 b_1^q + a_2 b_2^q + a_3 b_3^q\]
so that $\mathcal{H}$ is precisely the set of $\langle x \rangle$ such that $\sigma(x,x) = 0$. Since $\{a,b,d\},\{a,c,e\},\{c,d,f\}$ and $\{b,e,f\}$ are collinear triples,
we may choose generators for $a,b,c,d,e,f$ satisfying $d = a + b$, $e = a + c$ and $f = a + b + c$.
For convenience, write $a^q = (a_1^q,a_2^q,a_3^q)$. Let $A$ be the matrix whose rows are $a,b$ and $c$ and let $B$ be the matrix whose columns are $a^q,b^q,c^q$. Since $a,b,c$ are not collinear, $A$ is non-singular,  and since $x \mapsto x^q$ is a field automorphism, the matrix $B$ is also non-singular. Therefore
\[ AB = \begin{pmatrix}
	\sigma(a,a)  & \sigma(a,b)  & \sigma(a,c)  \\
	\sigma(b,a)  & \sigma(b,b)  & \sigma(b,c)  \\
	\sigma(c,a)  & \sigma(c,b)  & \sigma(c,c) \\
\end{pmatrix}\]

\bigskip

is also non-singular. On the other hand, since $\sigma(d,d) = 0$, it follows that $\sigma(a,b) = -\sigma(b,a)$, and similarly using $\sigma(e,e) = \sigma(f,f) = 0$, we find \[ \sigma(a,c) = -\sigma(c,a) \quad  \mbox{ and }\quad \sigma(b,c) = -\sigma(c,b).\]
Since $\sigma(a,a) = \sigma(b,b) = \sigma(c,c) = 0$, the diagonal of $AB$ is zero. Therefore
\[ \mbox{det}(AB) = \sigma(a,b)\sigma(b,c)\sigma(c,a) + \sigma(a,c)\sigma(c,b)\sigma(b,a) = 0\]
contradicting that $AB$ is non-singular.
\end{proof}

\bigskip

The bipartite incidence graph of the O'Nan configuration is shown in the figure below:
\begin{center}
\includegraphics[width=2in]{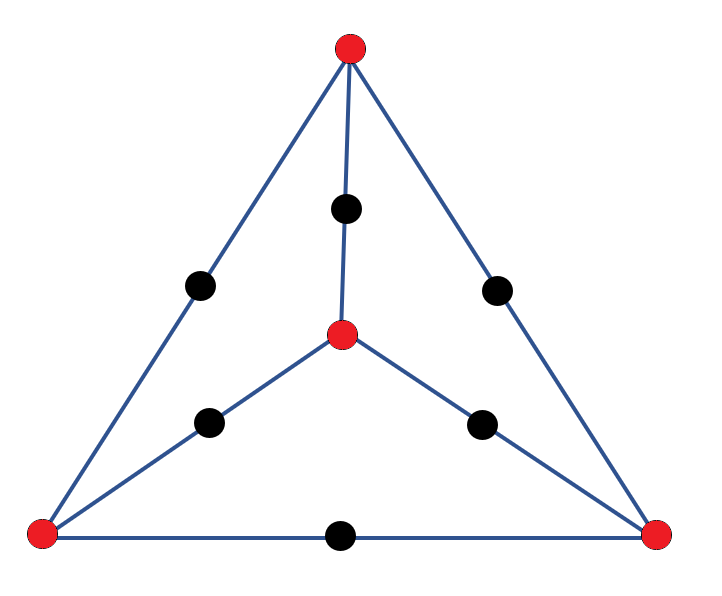}
\begin{center}
The 1-subdivision of $K_4$
\end{center}
\end{center}

In combinatorial terms, this bipartite graph is a {\em 1-subdivision} of $K_4$ -- the red points in the figure correspond to the lines.
The bipartite incidence graph $B_q$ of secants $L$ and points $\cal H$ has $q^2(q^2 - q + 1)$ vertices in $L$ and $q^3 + 1$ vertices in $\cal H$. The vertices in $L$ have degree $q + 1$ and the vertices in $\cal H$ have degree $q^2$, and consequently
$e(B_q) = q^2(q^3 + 1)$. A key property is that $B_q$ does not contain a 1-subdivision of $K_4$ with four vertices in $L$ and six vertices in $\cal H$, since the O'Nan configuration is absent from the Hermitian unital. 
The problem of finding the largest number of edges in a bipartite graph of given order not containing a fixed bipartite subgraph is known as a Zarankiewicz-type problem~\cite{Z}. The bipartite incidence graph $B_q$ is in this sense a near extremal bipartite graph not containing a 1-subdivision of $K_4$ with four vertices in $P$ and six in $\cal H$.
We refer the reader to Janzer~\cite{Janzer2019,Janzer2021}, Conlon, Janzer and Lee~\cite{CJL,ConlonLee} and Jiang and Qiu~\cite{JiangQiu2020,JiangQiu2023} for work on the extremal problem for subdivisions.

\subsection{The graph $H_q$}

We construct a graph $H_q$ on $L = V(H_q)$ in which $\{u,v\}$ is an edge if $u$ and $v$ intersect in a point of the unital, and we list basic properties of $H_q$ in this section. The graph $H_q$ is in fact a strongly regular graph; strongly regular graphs are the subject of extensive research in the literature,
and their properties can be found in the recent monograph of Brouwer and Van Maldeghem~\cite{BV22}. The graph $H_q$ is
denoted {\sf NU}$_3(q^2)$ in the literature (see page 81 in Brouwer and Van Maldeghem~\cite{BV22}).
We record the following lemma listing only the basic properties of $H_q$ we need -- these properties are all verified by O'Nan~\cite{onan} using elementary group theory and finite geometry:

\begin{proposition}\label{thm:basegraph}
The graph $H_q$ is an $n$-vertex $d$-regular graph with
\vspace{-0.25in}
\begin{center}
	\begin{tabular}{cp{6.2in}}
		$({\rm i})$ & $n = q^2(q^2-q+1) = q^4-q^3+q^2$ and $d = (q+1)(q^2-1) = q^3+q^2-q-1$, \\
		$({\rm ii})$ & a set $\cal C$ of $q^3 + 1$ maximal cliques of order $q^2$, every two sharing exactly one vertex,\\
        $({\rm iii})$ & each vertex in exactly $q + 1$ cliques of $\cal C$, \\
		$({\rm iv})$ & every copy of $K_4$ in $H_q$ contains at least three vertices in some clique in $\cal C$.
	\end{tabular}
\end{center}
\end{proposition}

{\bf Proof.}
	Combining the facts that there are $q^2+1$ lines through a point in $\mathrm{PG}(2,q^2)$ and lines intersect in $1$ or $q+1$ points, one can verify that there are $q^2$ secants and a unique tangent through any point in $\cal H$. In particular, there are $q^3+1$ tangents.

\bigskip

Proof of (i): As there are $q^4+q^2+1$ lines in $\mathrm{PG}(2,q^2)$, it follows from the remarks above that the number of secants is $n = q^4+q^2+1-(q^3+1)$. Given a secant $\ell$, there are $q^2-1$ more secants through every point of $\cal H$ on $\ell$, which shows $d = (q+1)(q^2-1)$.

\bigskip

Proof of (ii): Observe that the $q^2$ secants through a fixed point of $\cal H$ form a maximal clique in $H_q$. We define $\cal C$ to be the set of all such cliques, one for each point of $\cal H$. Since there is exactly one secant through two distinct points of $\cal H$, it follows that any two distinct cliques in $\cal C$ intersect in exactly one vertex.

\bigskip

Proof of (iii): As every secant contains $q+1$ points of $\cal H$, it follows that every vertex is contained in $q+1$ cliques in $\cal C$.

\bigskip

Proof of (iv): This follows from Proposition \ref{thm:no-onan}: the lack of an O'Nan configuration implies that every copy of $K_4$ in $H_q$ corresponds to four secants, at least three of which are concurrent in a point of $\cal H$ -- the three concurrent secants comprise a triangle in the clique in $\cal C$ corresponding to their intersection point.
$\Box$

\bigskip

\textbf{Remark.} With a bit more work and using Proposition \ref{thm:no-onan}, one can show that the set $\cal C$ is the set of all {\em maximum} cliques whenever $q \geq 3$. For $q = 2$ one can construct different cliques of size four by taking three secants concurrent in a point of $\cal H$ and one more secant intersecting each of the three others in a point of $\cal H$. Since we don't need the fact that the cliques in $\cal C$ are maximum, we will not refer to them as such in the remainder. \\

\textbf{Remark.} The graph $H_q$ also has many other interesting properties. For instance, it is a strongly regular graph in which adjacent vertices have $2q^2 - 2$ common neighbors and non-adjacent vertices have $(q + 1)^2$ common neighbors. The spectrum of its adjacency matrix is therefore determined by the theory of strongly regular graphs (see Section 3.1.6 in Brouwer and Van Maldeghem~\cite{BV22}). In this way, one can find that the non-trivial eigenvalues of the adjacency matrix of $H_q$ are $q^2-q-2$ and $-q-1$, with multiplicities $q^3$ and $(q^2-q-1)(q^2-q+1)$ respectively.

\bigskip

It will be convenient throughout the following sections to let $m = 2^{24}q^2$.

\subsection{Clique structure of $H_q$}\label{sec:cliquesection}

By Proposition \ref{thm:basegraph}.ii, $H_q$ is a union of maximal cliques of order $q^2$ pairwise intersecting in at most one vertex.
The goal of this section
is to prove the following lemma, which says in words that although many edges of $H_q[X]$ may lie in cliques of large size relative to $|X|$  -- as many as ${q^2 \choose 2}$ in a single clique -- in all instances of sets $X$ of size $m$, many edges  lie in cliques of linear size in $q$. The point is to show concentration of the number of edges in $X$ once each clique is replaced by a random complete bipartite graph. Fix a set $X \subseteq V(H_q)$, and let
\[\mathcal{T}_X = \{X \cap C \,\, | \,\, C \in \mathcal{C}, |X \cap C| \geq 2\},\] where $\cal C$ is the set of cliques from Proposition \ref{thm:basegraph}.ii.

\begin{lemma}\label{thm:cliques}
Let $X \subseteq V(H_q)$ with $|X| = m = 2^{24}q^2$. Then either the number of edges of $H_q[X]$ contained in cliques in $\mathcal{T}_X$ of
order at most $\sqrt{2m}/\log n$ is at least
\begin{equation}\label{small}
\frac{m^2}{64q}
\end{equation}
or the number of edges of $H_q[X]$ in cliques in $\mathcal{T}_X$ of order between $\sqrt{2m}/\log n$ and $\sqrt{2m}$ is at least
\begin{equation}\label{medium}
\frac{qm^{3/2}}{16\log^2 \! n}.
\end{equation}
\end{lemma}

It is convenient to introduce some further notation and terminology to prove this lemma.  We consider a partition of $\mathcal{T} = \mathcal{T}_X$ into three sets $\mathcal{S} \sqcup \mathcal{M} \sqcup \mathcal{L}$ of {\em small}, {\em medium} and {\em large} cliques, respectively, where if $|X| = k$,
\begin{eqnarray*} \
\mathcal{S} &=& \{T \in \mathcal{T} : 2 \leq |V(T)| \leq \sqrt{2k}/\log n\},  \\
\mathcal{M} &=& \{T \in \mathcal{T} : \sqrt{2k}/\log n < |V(T)| \leq \sqrt{2k}\},  \\
\mathcal{L} &=& \{T \in \mathcal{T} : \sqrt{2k} < |V(T)| \leq q^2\}.
\end{eqnarray*}
For a set $\mathcal{U} \subseteq \mathcal{T}$, it is convenient to define
\begin{eqnarray*}
v(\mathcal{U}) = \sum_{T \in \mathcal{U}} |V(T)| \quad \mbox{ and } \quad
e(\mathcal{U}) = \sum_{T \in \mathcal{S}} \textstyle{{|V(T)| \choose 2}}.
\end{eqnarray*}
In words, $e(\mathcal{U})$ is the number of edges in cliques in $\mathcal{U}$, and Lemma \ref{thm:cliques} claims if $|X| = m$, then $e(\mathcal{S}) \geq m^2/64q$ or $e(\mathcal{M}) \geq qm^{3/2}/16\log^2 \! n$. We use the following key lemma:

\begin{lemma}\label{lem:edgecount}
For any set $X \subseteq V(H_q)$,
\begin{eqnarray}
v(\mathcal{L}) &\leq& 2|X|, \label{c4bound} \\
v(\mathcal{S} \sqcup \mathcal{M}) &\geq& (q - 1)|X| - q^3 - 1. \label{c4bound2}
\end{eqnarray}
\end{lemma}

\begin{proof}
By Proposition \ref{thm:basegraph}.iii, if $|X| = k$ then $v(\mathcal{T}) \geq (q + 1)k - q^3 - 1$ since the number of cliques $T \in \mathcal{T}$
such that $|V(T) \cap X| = 1$ is at most $q^3 + 1$. So (\ref{c4bound}) implies (\ref{c4bound2})
using $v(\mathcal{S} \sqcup \mathcal{M}) = v(\mathcal{T})  - v(\mathcal{L})$.
Let $d(x) = |\{T \in \mathcal{L} : x \in V(T)\}|$ be the degree of $x$. As cliques in $\mathcal{T}$ pairwise share at most one vertex by Proposition \ref{thm:basegraph}.ii,
\[ \sum_{x \in X} {d(x) \choose 2} \leq {|\mathcal{L}| \choose 2}.\]
Applying Jensen's inequality, noting that the average degree is $\overline{d} = v(\mathcal{L})/k$, we obtain
\[ |\mathcal{L}|^2 \geq v(\mathcal{L})(\tfrac{v(\mathcal{L})}{k} - 1).\]
Suppose, for a contradiction, $v(\mathcal{L}) > 2k$. Then $v(\mathcal{L})/k - 1 > v(\mathcal{L})/2k$. Since each element of $\mathcal{L}$ is a large clique, $v(\mathcal{L}) > \sqrt{2k}|\mathcal{L}|$. Therefore
\[ |\mathcal{L}|^2 > \frac{v(\mathcal{L})^2}{2k} > \frac{(\sqrt{2k}|\mathcal{L}|)^2}{2k} = |\mathcal{L}|^2,\]
a contradiction. We conclude $v(\mathcal{L}) \leq 2k$.
\end{proof}

\bigskip

{\bf Proof of Lemma \ref{thm:cliques}.} By (\ref{c4bound2}), either $v(\mathcal{S}) \geq \frac{1}{2}[(q - 1)m - q^3 - 1]$ or
$v(\mathcal{M}) \geq \frac{1}{2}[(q - 1)m - q^3 - 1]$. We consider each of these cases separately, in order to prove
(\ref{small}) and (\ref{medium}) respectively.

\bigskip

{\bf Case 1.} $v(\mathcal{S}) \geq \frac{1}{2}[(q - 1)m - q^3 - 1]$. By Proposition \ref{thm:basegraph}.ii, $|\mathcal{S}| \leq q^3 + 1$. We apply Jensen's inequality to obtain
\begin{eqnarray*}
	e(\mathcal{S}) &=& \sum_{T \in \mathcal{S}} {|V(T)| \choose 2} \\
	&\geq& |\mathcal{S}| \cdot {v(\mathcal{S})/|\mathcal{S}| \choose 2} \\
	&\geq& \frac{v(\mathcal{S})\left(v(\mathcal{S})-|\mathcal{S}|\right)}{2} \\
	&\geq&  \frac{((q - 1)m - q^3 - 1)((q - 1)m - q^3 - 1-(q^3+1))}{8}
\end{eqnarray*}
using the lower bound on $v(\mathcal{S})$ and the upper bound on $|\mathcal{S}|$. A calculation using $m = 2^{24}q^2$ shows that this is at least $m^2/64q$ for all $q \geq 2$, proving (\ref{small}).

\bigskip

{\bf Case 2.}  $v(\mathcal{M}) \geq \frac{1}{2}[(q - 1)m - q^3 - 1]$. Since $|V(T)| \leq \sqrt{2m}$ for every clique $T \in \mathcal{M}$, and as $q \geq 2$ and $m \geq 2^{24}q^2$,
\begin{eqnarray*}
|\mathcal{M}| \cdot \sqrt{2m} \; \; \geq \; \; v(\mathcal{M}) &\geq& \frac{1}{2}[(q - 1)m - (q^3 + 1)] \\
&\geq& \frac{1}{\sqrt{8}}(q - 1)m.
\end{eqnarray*}
In particular, as $q \geq 2$, $|\mathcal{M}| \geq q\sqrt{m}/8$. As $|V(T)| \geq \sqrt{2m}/\log n$
for all $T \in \mathcal{M}$,
\[ e(\mathcal{M}) = \sum_{T \in \mathcal{M}} {|V(T)| \choose 2} \geq \displaystyle{\frac{q\sqrt{m}}{8} \cdot {\sqrt{2m}/\log n \choose 2}}.\]
A calculation using $m = 2^{24}q^2$ gives (\ref{medium}). This completes the proof of Lemma \ref{thm:cliques}. $\Box$

\section{The random $K_4$-free graph $H_q^*$}\label{sec:hqstar}

According to Proposition \ref{thm:basegraph}.ii, the graph $H_q$ is a union of maximal cliques of size $q^2$ pairwise intersecting in at most one vertex. For each maximal clique $T$ in $H_q$, let $(A_T,B_T)$ be a random partition
of $V(T)$ defined by independently placing vertices in $A_T$ or $B_T$ with probability $1/2$ each. Let $H_q^*$ be the random graph consisting of the union over all maximal cliques $T$ in $H_q$ of the complete bipartite subgraph with parts $A_T$ and $B_T$. According to Proposition \ref{thm:basegraph}.iv, this graph $H_q^*$ does not contain a complete graph of order four, since each complete graph of order four contains at least three vertices from some maximal clique $T$,
whereas $H_q^*[V(T)]$ is bipartite and therefore triangle-free for each maximal clique $T$. In this sense, $H_q^*$ is a random $K_4$-free graph. We plan to prove the following theorem:

\begin{theorem}\label{thm:pseudo}
For each prime power $q \geq 2^{40}$, there exists a $K_4$-free graph $G_q^*$ with $q^2(q^2 - q + 1)$ vertices such that
for every set $X$ of at least $m = 2^{24}q^2$ vertices of $G_q^*$,
\[ e(G_q^*[X]) \geq \frac{|X|^2}{256q}.\]
\end{theorem}

This theorem is essentially best possible, in the sense that the average set $X$ of vertices of $H_q$
of size at least about $q^2$ induces at most about $|X|^2/q$ edges.
The graph $G_q^*$ in Theorem \ref{thm:pseudo} will simply be an instance of the random graph $H_q^*$. We prove Theorem \ref{thm:pseudo} in Section \ref{proofofpseudo}, using the Hoeffding-Azuma martingale inequality~\cite{HA} in Section \ref{sec:prob}. It is possible to invoke
other concentration inequalities, such as McDiarmid's bounded differences inequality~\cite{Mc} or the Hanson-Wright~\cite{HW} inequality as in Conlon~\cite{C}, but
we opted to describe explicitly the fairly simple martingale which leads to the proof of Theorem \ref{thm:pseudo}.

\bigskip

{\bf Remark.} The approach used to prove Theorem \ref{thm:pseudo} can be used to improve the results of Conlon~\cite{C} on
triangle-free graphs: it was shown that if $J_q$ is the graph whose vertices are the points of a generalized quadrangle of order $q$ and
whose edges are the pairs of collinear points, then $J_q$ is a union of $q^3 + q^2 + q + 1$ edge-disjoint cliques of order $q + 1$.
Let $J_q^*$ be defined by taking random complete bipartite graphs in these cliques of order $q + 1$. Then for any set $X$ of vertices of $J_q^*$, writing the expected density of $J_q^*$ as $p = (1 + o(1))/2q$, Conlon~\cite{C} showed
\[ e(X) \geq p{|X| \choose 2} - O(q\log q) |X|.\]
This is effective for $|X| = \Omega(q^2\log q)$. Using the approach in this paper, we can show
\[ e(X) \geq p{|X| \choose 2} - O(q)|X|\]
which is effective for $|X| = \Omega(q^2)$, and is in this sense best possible, since $J_q$ has independent
sets of size $q^2 + 1$. The separation of small from medium cliques is key in eliminating the logarithmic factor in Conlon's bound, and
this solves a question raised in the concluding remarks of Conlon~\cite{C}. For $H_q^*$, the approach in this paper
may be used to give a similar lower bound on $e(X)$. However, for our purposes, it is more convenient to
use in $H_q^*$ a lower bound of the form $e(X) = \Omega(|X|^2/q)$ in order to prove Theorem \ref{thm:main}.

\subsection{Pseudorandomness in $H_q^*$}\label{sec:prob}

The main result of this section, which essentially says that sets of large quadratic size in $H_q^*$
induce many edges with high probability, is proved using the Hoeffding-Azuma Inequality~\cite{HA}, which may be stated in the following form:

\begin{proposition} {\bf (Hoeffding-Azuma Inequality)} \label{prop:ha}
Let $\lambda \geq 0$ and $c_1,c_2,\dots,c_k > 0$ be reals, and $Z = (Z_0,Z_1,Z_2,\dots,Z_k)$ be a martingale with $Z_0 = \mathbb E(Z)$ and $|Z_i - Z_{i - 1}| \leq c_i$ for all $i \leq k$. Then
\[ \mathbb P(Z - Z_0 \leq -\lambda) \leq \exp\Bigl(-\frac{2\lambda^2}{\sum_{i = 1}^k c_i^2}\Bigr).\]
\end{proposition}

If $c_i \leq c$ for all $i$, then the martingale as in Proposition \ref{prop:ha} is called {\em $c$-Lipschitz}.

\bigskip

For a set $X \subseteq V(H_q^*)$, it is convenient to define the random variable $Z_X = e(H_q^*[X])$.
The main result used to prove Theorem \ref{thm:pseudo} is the following, which essentially proves Theorem \ref{thm:pseudo}
for sets $X$ of size exactly $m$. Then Theorem \ref{thm:pseudo} follows for any set $X$ with $|X| \geq m$ by sampling
$m$ vertices of $X$ -- see Section \ref{proofofpseudo}.

\begin{lemma}\label{thm:concentration}
If $q \geq 2^{40}$ and $X \subseteq V(H_q)$ and $|X| = m = 2^{24}q^2$, then
\begin{equation}\label{probbound}
\mathbb P(Z_X \leq 2^{40}q^3) < n^{-m}.
\end{equation}
\end{lemma}

\begin{proof}
It is convenient to use the notation of Section \ref{sec:cliquesection}, specifically $\mathcal{S}$ and $\mathcal{M}$ denote the small and medium cliques in $H_q[X]$, and $e(\mathcal{S})$ and $e(\mathcal{M})$ denote the number of edges of $H_q[X]$ in those cliques, respectively.
By Lemma \ref{thm:cliques}, either $e(\mathcal{S}) \geq m^2/64q$ or $e(\mathcal{M}) \geq qm^{3/2}/16\log^2 \! n$ edges, and we consider these two cases separately.
Let $Z_{\mathcal{S}}$ and $Z_{\mathcal{M}}$ be the number of edges of $H_q^*$ in small and medium
cliques respectively, so $Z_X \geq Z_{\mathcal{S}} + Z_{\mathcal{M}}$. Recall $H_q^*$ is a union over maximal cliques $T \subseteq H_q$ of complete bipartite graphs with parts $A_T$ and $B_T$, where $A_T \sqcup B_T$ is a random partition of $V(T)$ such that vertices are placed independently in $A_T$ or $B_T$ with probability $1/2$ each, independently for each maximal clique $T$ of $H_q$.

\bigskip

{\bf Case 1.} $e(\mathcal{S}) \geq m^2/64q$.
For a small maximal clique $T$ of $H_q$, and a vertex $v \in V(T)$, let $Z_{v,T} = 0$ if $v$ is placed in $A_T$ and let $Z_{v,T} = 1$ if $v$ is placed in $B_T$, and these random variables are independent. Let $Z = Z_{\mathcal{S}}$ and let $Y_1,Y_2,\dots,Y_k$ be an ordering of all the random variables $Z_{v,T}$ for $T \in \mathcal{S}$ and $v \in V(T)$, and let $\mathcal{F}_i$ be the $\sigma$-field generated by $Y_1,Y_2,\dots,Y_i$. Here $k = v(\mathcal{S})$ and $Z$ is a function of $k$ independent random variables $Z_{v,T}$ for $T \in \mathcal{S}$. Then $Z_i = \mathbb E(Z \; | \; \mathcal{F}_i)$ defines a martingale which terminates with $Z$. If $c_i = \max|Z_{i} - Z_{i - 1}|$, then by the Hoeffding-Azuma inequality, for any $\lambda \geq 0$,
\[ \mathbb P(Z - \mathbb E(Z) \leq -\lambda) \leq \exp\Bigl(-\frac{2\lambda^2}{\sum_{i = 1}^k c_i^2}\Bigl).\]
Since $c_i \leq |V(T)| - 1$ when $Y_i = Z_{v,T}$,
\[ \sum_{i = 1}^k c_i^2 \leq \sum_{T \in \mathcal{S}} \sum_{v \in V(T)} (|V(T)| - 1)^2 = \sum_{T \in \mathcal{S}} |V(T)|(|V(T)| - 1)^2.\]
As each clique in $\mathcal{S}$ is a small clique,
\[ \max\{c_i : 1 \leq i \leq k\} \leq \max\{|V(T)| : T \in \mathcal{S}\} \leq \sqrt{2m}/\log n,\]
so the martingale is $c$-Lipschitz with $c = \sqrt{2m}/\log n$, and
\begin{eqnarray*}
\sum_{i = 1}^k c_i^2 &\leq& \sum_{T \in \mathcal{S}} |V(T)|(|V(T)| - 1)^2 \\
&\leq& c \cdot \sum_{T \in \mathcal{S}} (|V(T)| - 1)^2 \\
&\leq& 2c \cdot \sum_{T \in \mathcal{S}} \textstyle{{|V(T)| \choose 2}} \; \; = \; \; 2c \cdot e(\mathcal{S}).
\end{eqnarray*}
We conclude
\[ \mathbb P(Z - \mathbb E(Z) \leq -\lambda) \leq \exp\Bigl(-\frac{\lambda^2}{c \cdot e(\mathcal{S})}\Bigr).\]
We select $\lambda =  \mathbb E(Z)/2$. Since $e(\mathcal{S}) \geq m^2/64q$, and $\mathbb E(Z) = e(\mathcal{S})/2$,
\[ \lambda  = \tfrac{1}{4}e(\mathcal{S})  \geq \displaystyle{\frac{m^2}{256q}} = 2^{40}q^3.\]
Using $\lambda/e(\mathcal{S}) = 1/4$ and $\lambda/4c \geq e(\mathcal{S})/16c \geq m^{3/2}\log n/(4096q)$,
\[ \mathbb P(Z \leq \lambda) = \mathbb P(Z - \mathbb E(Z) \leq -\lambda) \leq \exp\Bigl(-\frac{\lambda}{4c}\Bigr) \leq  \exp\Bigl(-\frac{m^{3/2}\log n}{4096q}\Bigr) = n^{-m}\]
where we used $m = 2^{24}q^2$. This completes Case 1.

\bigskip

{\bf Case 2.} $e(\mathcal{M}) \geq qm^{3/2}/16\log^2 \! n$. Let $Z = Z_{\mathcal{M}}$ and let
\[ \lambda = \frac{1}{2}\mathbb E(Z) = \frac{1}{4} e(\mathcal{M}).\]
Then $\lambda \geq 2^{40}q^3$ since $m = 2^{24}q^2$ and $q \geq 2^{40}$.
By the Hoeffding-Azuma inequality as in Case 1, with $c = \sqrt{2m}$ for the medium cliques in $\mathcal{M}$, we obtain
\begin{eqnarray*}
\mathbb P(Z \leq \lambda) &\leq& \exp\Bigl(-\frac{\lambda}{4c}\Bigr) \\
&\leq& \exp\Bigl(-\frac{qm}{1024\log^2\! n}\Bigr) \;\; \leq \; \;n^{-m}
\end{eqnarray*}
as $q \geq 2^{40}$ easily implies $q \geq 1024\log^3\! n$.
\end{proof}

\subsection{Proof of Theorem \ref{thm:pseudo}}\label{proofofpseudo}

Let $m = 2^{24}q^2$. Then, by Lemma \ref{thm:concentration}, the probability that a set $X \subseteq V(H_q^*)$ of size $m$
induces at most $2^{40}q^3$ edges of $H_q^*$ is at most $n^{-m}$. It follows that the expected number of such $X$ is at most ${n \choose m}n^{-m} < 1$,
and so there exists an instance $G_q^*$ of $H_q^*$ such that every set of $m$ vertices induces at least $2^{40}q^3$ edges. Fix such a $G_q^*$, and recall
$G_q^*$ is $K_4$-free with $q^2(q^2 - q + 1)$ vertices.  To prove Theorem \ref{thm:pseudo}, consider a set $X$ of at least $2^{24}q^2$ vertices. We count pairs $(e,Y)$ where $Y \subseteq X$ has size $m$ and $e$ is an edge of $G_q^*[Y]$.
On one hand, the number of such pairs is at least
\[ {|X| \choose m} \cdot 2^{40}q^3\]
by the choice of $G_q^*$, whereas the number of pairs $(e,Y)$ is also exactly
\[ e(G_q^*[X]) \cdot {|X| - 2 \choose m - 2}.\]
We conclude that for every set $X$ of size at least $m$ in $G_q^*$,
\begin{eqnarray*}
 e(G_q^*[X]) &\geq& 2^{40}q^3 \frac{|X|^2}{m^2} \\
 &\geq& \frac{|X|^2}{256q}. \end{eqnarray*}
This proves Theorem \ref{thm:pseudo}. $\Box$

\section{Randomly sampling from $G_q^*$}\label{sec:proofofmain}

In this section, we prove Theorem \ref{thm:main} using the graph $G = G_q^*$ guaranteed by Theorem \ref{thm:pseudo}, a theorem in Section \ref{sec:container} for counting independent sets,
and random sampling as in~\cite{MV}. The theorem on counting independent sets is based on early work of Kleitman and Winston~\cite{KW}.

\subsection{Counting independent sets}\label{sec:container}

The following is found in Kohayakawa, Lee, R\"{o}dl and Samotij~\cite{KLRS}, and is a special case of the
method of {\em containers} due to Balogh, Morris and Samotij~\cite{BMS} and Saxton and Thomason~\cite{ST}:

\begin{proposition}\label{thm:containerforgraphs}
	Let $G$ be a graph on $n$ vertices, and let $r,R \in \mathbb{N}$, and $\alpha \in [0,1]$  satisfy:
	\begin{equation}
		e^{-\alpha r}n \leq R, \label{eq:graphcondition1}
	\end{equation}
	and, for every subset $X \subseteq V(G)$ of at least $R$ vertices,
	\begin{equation}
		2e(X) \geq \alpha|X|^2. \label{eq:graphcondition2}
	\end{equation}
	Then the number of independent sets of size $t \geq r$ in $G$ is at most
	\begin{equation}\label{eq:finalcount}
 {n \choose r} {R \choose t - r}.
 \end{equation}
\end{proposition}

\medskip
For completeness, we briefly outline the proof of this result: to count independent sets,
we may select the vertices of an independent set one by one, and then delete their neighbors. If $X$ is the set of
vertices remaining in $G$ at any particular stage in the process, and $|X| \geq R$, then there exists a vertex $v$ in $X$
with at least $\alpha |X|$ neighbors in $X$, due to (\ref{eq:graphcondition2}). Removing $v$ and its neighbors, we
arrive at a new set $X'$ of remaining vertices satisfying $|X'| \leq (1 - \alpha)|X|$. If we repeat this $r$ times,
there are at most $(1 - \alpha)^r n \leq e^{-\alpha r}n \leq R$ vertices left to select from, and at that stage
we select the remaining $t - r$ vertices of the independent set.

\subsection{The proof of Theorem \ref{thm:main}}

Let $G = G_q^*$ be the graph guaranteed by Theorem \ref{thm:pseudo}, so $G$ is a $K_4$-free graph with $n = q^2(q^2 - q + 1)$ vertices such that
for every set $X$ of at least $2^{24}q^2$ vertices of $G$, $e(X) \geq |X|^2/256q$, for each large enough prime power $q$. This allows us to apply Proposition \ref{thm:containerforgraphs} to $G$.

\bigskip

The main claim is that the number of independent sets of size $t = 2^{30}q\log^2\!q$ in $G$ is at most $(q/\log^2 \! q)^t$.
To prove this, let $R = 2^{24}q^2$, $r = 1024q\log q$, and $\alpha = 1/256q$ in Proposition \ref{thm:containerforgraphs}.
Then $\exp(-\alpha r) n = e^{-4\log q} n \leq q^{-4}n \leq R$ so (\ref{eq:graphcondition1}) is satisfied.
Since $e(X) \geq \alpha |X|^2$ for all $X \subseteq V(G)$ with $|X| \geq R$, (\ref{eq:graphcondition2}) is also satisfied.
By (\ref{eq:finalcount}),  the number of independent sets of size $t$ in $G$ is at most
\[ {n \choose r} {R \choose t - r} \leq n^r {R \choose t} \leq q^{4r} \Bigl(\frac{4R}{t}\Bigr)^t\]
using the bound ${x \choose y} \leq (4x/y)^y$ for integers $x \geq y \geq 1$, and using $n \leq q^4$.
Using $q^{4r} \leq e^{t/2} \leq 2^t$ and $4R/t \leq q/2\log^2 \! q$, this is at most $(q/\log^2 \! q)^t$, which proves the claim.

\bigskip

Finally, we prove Theorem \ref{thm:main} using the claim. Randomly sample a set $V$ of vertices of $G$ with probability $\log^2 \! q/q$ independently for each vertex. If $I$ is the number of independent sets of size $t$ in $G[V]$, then $\mathbb E(I) \leq 1$ and therefore
\[ \mathbb E(|V| - I) \geq \frac{n\log^2 \! q}{q} - 1 \geq \frac{q^3 \log^2 \! q}{2}\]
since $n = q^2(q^2 - q + 1) \geq q^4/2 + q$. In particular, there exists $V \subseteq V(G)$ such that $G[V]$ is a $K_4$-free graph on at least $q^3\log^2 \! q/2$ vertices containing no independent set of size at least $t$. Since this is valid for any large enough prime $q$, and there is a prime $q$ between
any positive integer and its double by Bertrand's Postulate, this shows that there exists an absolute constant $c_1 > 0$ such that
 $r(4,t) \geq c_1 t^3/\log^4 \! t$ for all $t \geq 3$, proving Theorem \ref{thm:main}. $\Box$

 \section{Proof of Theorem \ref{thm:main2}}\label{sec:proofofmain2}

We use the random blowup approach of Kim and Mubayi alluded to in the work of Alon and R\"{o}dl~\cite{AR}.
Let $G_t$ denote a $K_4$-free graph with no independent sets of size at least $s = \lfloor t/\log t\rfloor$, where
$|V(G_t)| = T = \Omega(t^3/\log^7\! t)$, guaranteed by Theorem \ref{thm:main}. The {\em $r$-blowup} $G_t(r)$ of $G_t$ is the graph obtained
by replacing each vertex $x$ of $G_t$ with an independent set $I_x$ of size $r$  and each edge $\{x,y\}$ of $A_t$ by a complete bipartite graph between $I_x$ and $I_y$.
We shall set
\[ r = \Big\lceil \frac{\delta_k t^{2(k - 2)}}{(\log t)^{6k-13}}\Big\rceil\]
where $\delta_k > 0$ is a constant to be chosen shortly.
Alon and R\"{o}dl~\cite{AR} observed that the number of independent sets of size $t$ in $G_t(r)$ is at most
\[ \frac{{T \choose s}(s r)^t}{t!}.\]
For a permutation $\sigma$ of $V(G_t(r))$, let $G_t(r,\sigma)$ denote the copy of $G_t(r)$ with vertex set $\sigma(V(G_t(r)))$. For $k \geq 3$, taking
$k - 1$ independent random permutations $\sigma_1,\sigma_2,\dots,\sigma_{k - 1}$ of $V(G_t)$, let $G(k)$ be the graph with vertex set $V(G_t(r))$ and edge set
\[ E(G(k)) = \bigcup_{i = 1}^k E(G_t(r,\sigma_i)).\]
Each $G_t(r,\sigma_i)$ plays the role of the $i$th color in a $k$-coloring of the edges of $K_{rT}$, and the edges of $E(K_{rT}) \backslash E(G(k))$
form the last color. In the event that an edge is in more than one of the graphs $G_t(r,\sigma_i)$ -- in other words there is a choice of colors for the edge -- we arbitrarily assign one of the colors to the edge. The expected number of independent sets of $t$ vertices in $G(k)$ is precisely
\[ \Bigl(\frac{{T \choose s}(s r)^{t}}{t!}\Bigr)^{k - 1} {rT \choose t}^{-(k - 2)}.\]
Using $T \leq t^3$ and the bounds
\[ {T \choose s} \leq T^s \leq t^{3s} \leq e^{3t} \quad \mbox{ and }\quad {rT \choose t} \geq (rT/t)^t \geq (crt^2/\log^7\!t)^t\]
for some constant $c > 0$
and $t! \geq (t/e)^t$, the above expression is at most
\[   C_k^t \cdot \Bigl(\frac{sr}{t}\Bigr)^{(k - 1)t} \cdot \Bigl(\frac{rt^2}{\log^7\! t}\Bigr)^{-(k - 2)t}
\leq C_k^t \cdot r^t \cdot t^{-2(k - 2)t} \cdot (\log t)^{(6k - 13)t}\]
for some constant $C_k > 0$. If $\delta_k < 1/C_k$, then the choice of $r$ ensures this quantity is less than 1. Consequently, there exists a graph $G(k)$ with $rT$ vertices that is the union
of $k - 1$ copies of $K_4$-free graphs and $G(k)$ has no independent set of size $t$. Consequently, for each $k \geq 3$, there exists $\gamma_k > 0$ such that
\[ r_k(4;t) > rT \geq \gamma_k \frac{t^{2k - 1}}{(\log t)^{6(k - 1)}}.\]
This completes the proof of Theorem \ref{thm:main2}. $\Box$

\section{Concluding remarks}

$\bullet$ {\bf Asymptotics of $r(4,t)$.} To prove $r(4,t) = \Omega(t^3/\log^2\! t)$, it would be enough to prove that for some constants $c, C > 0$, the number of independent sets of size $t = cq\log q$ in $G_q^*$ is at most $(Cq^2/t)^t$. However, the technical condition (\ref{eq:graphcondition1}) precludes
an application of Proposition \ref{thm:containerforgraphs}. This loss of logarithmic factors appears to occur also in~\cite{MV} when counting
independent sets using spectral methods. A survey of counting independent sets in graphs is given by Samotij~\cite{Sam}.
Nevertheless, we believe  $r(4,t)$ has order $t^3/\log^2\! t$, and that $G_q^*$ may indeed have at most $(Cq^2/t)^t$ independent sets of size $t = cq\log q$ for some
constants $c,C > 0$:

\begin{conjecture}
As $t \rightarrow \infty$,
\[ r(4,t) = \Theta\Bigl(\frac{t^3}{\log^2\! t}\Bigr).\]
\end{conjecture}

\bigskip

$\bullet$ {\bf Spectral approach.} A key part of the proof of Theorem \ref{thm:main} is the pseudorandomness of the graphs $H_q^*$, as stated in Theorem \ref{thm:pseudo}. The non-trivial eigenvalues of the adjacency matrix of the graph $H_q$ are $q^2-q-2$ and $-q-1$ with multiplicities $q^3$ and $(q^2-q-1)(q^2-q+1)$ respectively. A one-sided version of the expander mixing lemma~\cite{Alon}, see for instance Theorem 3.5 in Haemers~\cite{Haemers}, shows
\[ 2e(X) \geq \frac{|X|^2}{2q} - (q + 1)|X|\]
for all sets $X \subseteq V(H_q)$, and so $H_q$ itself is pseudorandom. Unfortunately, as pointed out to us by Carl Schildkraut, the smallest eigenvalue of the random subgraph $H_q^*$ is of order at most $-q^2$ with high probability: a typical clique in $H_q$ is partitioned into two roughly equal parts, and the complete bipartite graph with those parts has smallest eigenvalue of order $-q^2$, which by Cauchy interlacing shows the same for $H_q^*$.

\bigskip

$\bullet$ {\bf Optimal pseudorandom graphs.} Another salient open problem is to determine whether there exists
a $K_4$-free $(n,d,\lambda)$-graph with $\lambda = O(\sqrt{d})$ and $d = \Omega(n^{4/5})$; this problem remains open. Such a graph would be optimal in the sense of the bounds of Sudakov, Szab\'{o} and Vu~\cite{SSV}.  As shown by Mubayi and the second author~\cite{MV}, this would imply the same
lower bound on $r(4,t)$ as in Theorem \ref{thm:main}. The graph $H_q$ is not an optimal pseudorandom graph as the non-trivial eigenvalues of the adjacency matrix of $H_q$ are $q^2-q-2$ and $-q-1$, with multiplicities $q^3$ and $(q^2-q-1)(q^2-q+1)$ respectively. These are determined from the theory of strongly regular graphs -- see Brouwer and Van Maldeghem~\cite{BV22}. Since $H_q$ is not an optimal pseudorandom graph, or even a weakly optimal pseudorandom graph in the sense of He and Wigderson~\cite{HeW}, it seems unlikely that $H_q^*$ is an optimal or weakly optimal pseudorandom graph; moreover $H_q^*$ is generally not a regular graph.

\bigskip

$\bullet$ {\bf Graph Ramsey numbers.} Given a graph $F$, define $r(F,t)$ to be the minimum $n$ such that every $F$-free $n$-vertex graph contains an independent set
of size $t$. In this paper, we studied $r(4,t)$ -- the case $F = K_4$ -- but we believe that the methods of our paper could be fruitful for other graph Ramsey numbers
$r(F,t)$. For instance, for the well-studied cycle-complete graph Ramsey numbers $r(C_k,t)$ when $k$ is odd -- see Sudakov~\cite{Sudakov} for upper bounds. There are a number of dense constructions of
$C_k$-free graphs with many rich properties to which the methods of this paper -- in particular, counting independent sets via containers -- may apply, such as the graphs defined by Margulis~\cite{Margulis}, the {\em Ramanujan graphs} of Lubotzky, Phillips and Sarnak~\cite{LPS}, the high girth graphs of Lazebnik, Ustimenko and Woldar~\cite{LUW}, and graphs from generalized polygons -- see~\cite{MV}. A survey of extremal problems for cycles in graphs is given in~\cite{V}. Recent work
of Conlon, Mubayi and the authors~\cite{CMMV} gives good lower bounds on cycle-complete graph Ramsey numbers.
We also believe the approach in this paper should also help with estimates of the {\em Erd\H{o}s-Rogers functions}~\cite{DR,DRR,W,KMV,C,MV,GJ} -- recent progress was obtained along these lines
by Janzer and Sudakov~\cite{JS} using the results of this paper.

 \bigskip

$\bullet$ {\bf Extremal graph theory.} The bipartite incidence graph $B_q$ of points $P$ in $\cal H$ and secants $L$ to the Hermitian unital $\cal H$ is in fact a near extremal bipartite graph not containing a 1-subdivision of $K_4$ with four vertices in $P$ and six in $L$ -- this is a {\em Zarankiewicz-type problem}~\cite{Z}. Let $F_4$ be a 1-subdivision of $K_4$, and define
$z(n,m,F_4)$ to be the maximum number of edges in an $n$ by $m$ bipartite graph which does not contain $F_4$ with four vertices in the part of size $n$ and six in the part of size $m$.  It is possible via counting arguments (adapting those of Conlon, Janzer and Lee~\cite{CJL}) to show that for $n \geq m \geq 2$,
\[ z(n,m,F_4) = O(m^{3/5}n^{4/5}).\]
When $G = B_q$, we have $m = q^3 + 1$ and $n = q^2(q^2 - q + 1)$ and $e(G) = (q + 1)n = q^5 + o(q^5)$, which matches the order of magnitude of the upper bound on $z(n,m,F_4)$ as $q \rightarrow \infty$. The extremal problem for subdivisions has been studied extensively -- see Conlon, Janzer and Lee~\cite{CJL} and the references therein. One of their conjectures implies that if $F_s$ is a 1-subdivision of $K_s$, then $z(n,n,F_s) = \Theta(n^{3/2 - 1/(4s - 6)})$, and this conjecture
remains open for $s \geq 4$. If this conjecture is true, and in addition extremal graphs contain no cycles of length four, then as in the work of Conlon, Mubayi and the authors~\cite{CMMV}, $r(s,t) = \widetilde{\Theta}(t^{s - 1})$ as $t \rightarrow \infty$. It may therefore be more fruitful for lower bounds on $r(s,t)$ for $s \geq 5$ to look for lower bounds on $z(m,n,F_s)$ i.e.\ for a near extremal $m$ by $n$ bipartite graph with no 1-subdivision of $K_s$, for appropriate values of $m$ and $n$, as was done in this paper for $s = 4$.

\bigskip
\bigskip

{\bf Acknowledgements.} We would like to acknowledge helpful comments from Anurag Bishnoi, Boris Bukh, David Conlon, Christian Elsholtz, Simon Griffiths, Ferdinand Ihringer, Yusheng Li, Dhruv Mubayi, Jiaxie Nie, Carl Schildkraut, and Lutz Warnke, and we are grateful to an anonymous referee for detailed comments which helped improve the paper.

\end{document}